\documentclass[11pt]{amsart}
\usepackage{amssymb}
\usepackage{amscd}
\usepackage[all]{xypic}
\usepackage{comment}

\numberwithin{equation}{section}

\def\today{\number\day\space\ifcase\month\or   January\or February\or
   March\or April\or May\or June\or   July\or August\or September\or
   October\or November\or December\fi\   \number\year}

\theoremstyle{definition}
\newtheorem{thm}{Theorem}[section]
\newtheorem{lem}[thm]{Lemma}
\newtheorem{prp}[thm]{Proposition}
\newtheorem{dfn}[thm]{Definition}
\newtheorem{cor}[thm]{Corollary}

\newtheorem{rmk}[thm]{Remark}

\newtheorem{exa}[thm]{Example}

\newtheorem{qst}[thm]{Question}

\newcommand{\beq}{\begin{equation}}
\newcommand{\eeq}{\end{equation}}
\newcommand{\beqr}{\begin{eqnarray*}}
\newcommand{\eeqr}{\end{eqnarray*}}
\newcommand{\bal}{\begin{align*}}
\newcommand{\eal}{\end{align*}}
\newcommand{\bei}{\begin{itemize}}
\newcommand{\eei}{\end{itemize}}

\newcommand{\af}{\alpha}
\newcommand{\bt}{\beta}
\newcommand{\gm}{\gamma}

\newcommand{\ep}{\varepsilon}

\newcommand{\Z}{{\mathbb{Z}}}
\newcommand{\R}{{\mathbb{R}}}
\newcommand{\C}{{\mathbb{C}}}
\newcommand{\N}{{\mathbb{Z}}_{> 0}}
\newcommand{\Nz}{{\mathbb{Z}}_{\geq 0}}

\newcommand{\id}{{\operatorname{id}}}

\newcommand{\diag}{{\operatorname{diag}}}

\newcommand{\Aut}{{\operatorname{Aut}}}
\newcommand{\Ad}{{\operatorname{Ad}}}

\newcommand{\tr}{{\operatorname{tr}}}

\newcommand{\RD}{{\mathrm{dim}}_{\mathrm{Rok}}^{\mathrm{c}}}

\newcommand{\OT}{{\mathcal{O}}_{2}}

\newcommand{\andeqn}{\qquad {\mbox{and}} \qquad}

\newcommand{\ifo}{if and only if}

\newcommand{\ca}{$C^*$-algebra}

\newcommand{\I}{\infty}

\newcommand{\Ann}{\mathrm{Ann}}
\newcommand{\indG}{\mathrm{ind}_G}

\title[Values of Rokhlin dimension]{Values of Rokhlin dimension for actions of compact  groups}   

\author{Ilan Hirshberg}
\author{N.~Christopher Phillips}

\date{27 July 2023}

\address{Department of Mathematics, Ben-Gurion University of the Negev,
	Be'er Sheva, Israel.}

\address{Department of Mathematics, University  of Oregon,
       Eugene OR 97403-1222, USA.}
   
   \makeatletter
   \@namedef{subjclassname@2020}{%
   	\textup{2020} Mathematics Subject Classification}
   \makeatother

\thanks{This material is based upon work partially supported by the
 Simons Foundation Collaboration Grant for Mathematicians \#587103, by the US National Science Foundation under Grant DMS-2055771 and by the US-Israel Binational Science Foundation.}

\begin{document}

\begin{abstract}
We show that any finite group admits actions on simple AF algebras with unique trace which have arbitrarily large finite values of Rokhlin dimension with commuting towers. We show similar results for actions of compact Lie groups, with AH algebras with no dimension growth in place of AF algebras. We also relate Rokhlin dimension to the $G$-index for actions of compact Lie groups on commutative $C^*$-algebras.
\end{abstract}

\maketitle

\section{Introduction}
Rokhlin dimension for actions of finite groups and of $\Z$ on $C^*$-algebras was introduced in \cite{HWZ}, as a generalization of the Rokhlin property. This notion has been generalized to compact group actions (\cite{Gardella-compact}), residually finite groups (\cite{SWZ}) and flows (\cite{HSWW}). The main motivation for studying Rokhlin dimension was that it serves as a regularity property for actions which is very useful as a way to establish various permanence properties for crossed products. 

Rokhlin dimension comes in two versions, with and without commuting towers. 
When introduced, the distinction came about as a technicality, as the stronger 
commuting tower assumption was needed for some permanence results but not for 
others; it was not immediately clear whether those properties are different. 
For finite groups, it turned out that they indeed are. For example, it was 
shown in \cite[Corollary 4.8]{HP15} that there are no actions of finite groups 
on the Jiang-Su algebra $\mathcal{Z}$ or the Cuntz algebra 
$\mathcal{O}_{\infty}$ which have finite Rokhlin dimension with commuting 
towers. However, without the commuting tower assumption, all outer actions of 
$\Z_2$ on $\mathcal{O}_{\infty}$ have Rokhlin dimension 1 (\cite[Theorem 
3.3]{BEMSW}) and any strongly outer action of a finite group on $\mathcal{Z}$ 
has Rokhin dimension at most 2 (\cite[Theorem A]{GHV}). 
(We note in passing that the situation for actions of $\Z$ appears to be different; see \cite[Theorem 6.2]{wouters2023equivariant}.)  These examples showed that the two notions of Rokhlin dimension are different, in that one is finite and one is not. Example 4.28 of \cite{GdHbSg} provides an action of $\Z_2$ on a simple AF algebra such that its Rokhin dimension with commuting towers is exactly 2. In \cite[Example 4.29]{GdHbSg}, building on the previously mentioned example, it is shown that there exists an action of $\Z_2$ on a Kirchberg algebra such that its Rokhlin dimension with commuting towers is 2, and its Rokhlin dimension without commuting towers is 1. This thus far is the only known example in which both dimensions are finite but are different.

It was shown in \cite[Theorem A]{GHV} that strongly outer actions of finite groups on simple nuclear unital separable $\mathcal{Z}$-stable $C^*$-algebras whose trace space is a Bauer simplex with finite dimensional extreme boundary have Rokhlin dimension at most $2$ without commuting towers.  Thus, the only possible values lie in $\{0,1,2\} \cup \{\infty\}$. We do not know whether the value $2$ can in fact occur. This dimension collapse phenomenon mirrors the dimension collapse phenomenon for nuclear dimension, proved in \cite[Corollary C]{CETWW}. 

The motivation for the current work comes from \cite[Examples 4.28, 4.29]{GdHbSg}. We show that for any finite group we can construct actions on simple AF algebras with unique trace which have arbitrarily large but finite Rokhlin dimension with commuting towers. We obtain similar results for non-discrete compact Lie group actions, but then we get AH algebras with no dimension growth. We find sharper bounds for actions of groups which have $\Z_2$ or $S^1$ as factors. This shows that the dimension collapse phenomenon does not occur in this setting, and opens up possible follow-up questions. We outline a few here.
\begin{qst}
	 Given a $C^*$-algebra $A$ and a compact Lie group $G$, what values of Rokhlin dimension with commuting towers can occur for actions of $G$ on $A$?  
\end{qst}
We know that no finite value can occur for $A = \mathcal{Z}$ or $A = \mathcal{O}_{\infty}$. It was shown in \cite[Theorem 4.34]{GdHbSg} that the value $1$ cannot occur for actions of $\Z_2$ on separable unital $C^*$-algebras which tensorially absorb the $2^{\infty}$ UHF algebra, but it is not known whether any other finite non-zero value can occur. Are there examples in which this set of numbers is an infinite proper subset of $\Nz$? Are there examples of pairs of $C^*$-algebras $A$ and finite groups $G$ for which all natural numbers occur?
Is this set of values an interesting invariant for the $C^*$-algebra $A$?
As a specific test question, one might ask the following.
\begin{qst}
	Does there exist a unital separable $C^*$-algebra $A$ which absorbs the $2^{\infty}$ UHF algebra tensorially and an action of $\Z_2$ on $A$ which has finite but strictly positive Rokhlin dimension with commuting towers?
\end{qst}
The fact that the value of Rokhlin dimension with commuting towers can be arbitrarily large means that this dimension can be viewed as a meaningful cocycle conjugacy invariant of the group action, and, in a sense, a measure of its complexity. It thus makes sense to try to conduct a finer study of actions with specific values, rather than find general theorems which follow just from the fact that the action is finite. 
\begin{qst}
A classification of actions of finite groups on UCT Kirchberg algebras with the Rokhlin property was done in \cite{Izm2}. Can one find a similar classification for actions of finite groups which have Rokhlin dimension 1 with commuting towers? If so, one would then ask about Rokhlin dimension 2, etc.
\end{qst}
\begin{qst}
	Find $K$-theoretic or other properties of crossed products which are implied by having specific value of Rokhlin dimension. 
\end{qst}
For example, suppose $G = \Z_2$ and suppose $A$ is an AF algebra. In a computation in the middle of \cite[Example 4.28]{GdHbSg}, it is shown that the crossed product of $A$ by an action $\alpha$ of $\Z_2$ with $\RD(\alpha) \leq 1$ has torsion free $K$-theory, whereas for the example in question, the torsion group of $K_0 (A \rtimes_{\alpha} \Z_2)$ is $\Z_2$. As a finer concrete follow-up question, one could ask whether having Rokhlin dimension $n$ implies some restriction on the order of torsion in the $K$-theory of the crossed product.

In Section~\ref{Sec_Basics}, we describe some notation, and prove the basic 
technical theorems on equivariant $K$-theory which are needed for the 
constructions. Our main examples are covered in Section~\ref{Sec_Exs}. In 
Section~\ref{sec:tensor-product-actions} we discuss tensor products and a 
K{\"u}nneth-type formula, which is used to construct examples of actions of 
groups which have $S^1$ or $\Z_2$ as a factor. Section~\ref{sec:commutative} is 
somewhat unrelated to the rest of the paper. We show there that all values of 
finite Rokhlin dimension can occur for actions on commutative $C^*$-algebras, 
by showing that Rokhlin dimension coincides with the so-called $G$-index of an 
action, shifted by $1$; this relies on an equivariant semiprojectivity argument.

\section{Preliminaries and results on which the constructions rely}\label{Sec_Basics}

In this paper all groups are compact (and often finite).
Lie groups are not required to be connected unless explicitly stated;
thus, all finite groups are Lie groups. The notation 
$\Z_n$ is used for the finite cyclic group $\Z / n \Z$, not the $n$-adic 
integers. When we say that an action $\alpha$ of a locally compact group $G$ on 
a a $C^*$-algebra $A$ is continuous we mean that $\alpha$ is point-norm 
continuous, that is, for any $a \in A$, the map $G \to A$ given by $g \mapsto 
\alpha_g(a)$ is continuous. We denote by $R(G)$ the representation ring of $G$. 
We denote by $I(G)$ the augmentation ideal.

	Let $A$ be a separable $C^*$-algebra. We take $A_{\infty} = l^{\infty}(A)/c_0(A)$ and view $A$ as embedded in $A_{\infty}$ as the image of the constant sequences. We denote by $\Ann(A) \subseteq A_{\infty} \cap A'$ the annihilator of $A$. 
	  This is an ideal in the central sequence algebra $A_{\infty} \cap A'$, and we define $F_{\infty}(A) = ( A_{\infty} \cap A' ) / \Ann(A)$. This distinction is relevant only when $A$ is non-unital: if $A$ is unital then $\Ann(A) = 0$. The $C^*$-algebra $F_{\infty}(A)$ is necessarily unital even if $A$ is not. We refer the reader to \cite{Kirchberg-central} for more details on these algebras. We restrict our attention to unital $C^*$-algebras here, in which case $F_{\infty}(A)$ is simply identified with $A_{\infty} \cap A'$. We nonetheless use this notation here as it is occurs in references.
	
	Suppose $G$ is a locally compact Hausdorff group, and let $\alpha \colon G 
	\to \Aut(A)$ be a continuous action. The action $\alpha$ induces a 
	coordinate-wise action on $l^{\infty}(A)$. This leaves  $c_0(A)$ invariant, 
	and therefore induces an action on $A_{\infty}$, which leaves the central 
	sequence algebra invariant; furthermore, the ideal $\Ann(A)$ is invariant, 
	so we obtain an induced action on $F_{\infty}(A)$. However, the action on 
	$l^{\infty}(A)$, and therefore on $F_{\infty}(A)$, need not be continuous. 
	We denote by $F_{\infty}^{(\alpha)}(A)$ the subalgebra of elements on which 
	$G$ acts continuously. By \cite[Theorem 11]{Brown-continuity}, this 
	coincides with the image  in the quotient of the subalgebra 
	$l^{\infty,(\alpha)}(A)$ of points on which $G$ acts continuously. 
	Likewise, we denote by $A_{\infty}^{(\alpha)}$ the subalgebra of elements 
	in $A_{\infty}$ on which $G$ acts continuously. 
	
	The following is taken from \cite[Definition 1.6]{HP15}. 

\begin{dfn}
	Let $G$ be a second countable compact Hausdorff group, and let $X$ be a second countable compact Hausdorff space endowed with a free action of $G$. Let $A$ be a separable $C^*$-algebra, and let $\alpha \colon G \to \Aut(A)$ be a continuous action. We say that $\alpha$ has the $X$-Rokhlin property if there exists a unital equivariant homomorphism from $C(X)$ to $F_{\infty}^{(\alpha)}(A)$. 
\end{dfn}

\begin{rmk}
	Suppose $A$ is unital. We find it more convenient to use the following equivalent characterization of the $X$-Rokhlin property: the action $\alpha$ has the $X$-Rokhlin property if and only if for any finite $F \subseteq A$, for any finite $S \subseteq C(X)$, and for any $\ep>0$, there exists a $G$-equivariant unital completely positive map $\varphi \colon C(X) \to A$ such that: 
	\begin{enumerate}
		\item $\| \varphi(a)\varphi(b) - \varphi(ab) \|<\ep$ for all $a,b \in S$.
		\item $\| [ \varphi(a) , x ] \| < \ep$ for all $a \in S$ and for all $x \in F$.
	\end{enumerate}
	This characterization follows easily from the Choi-Effros Lifting theorem (which applies since $C(X)$ is abelian). 
\end{rmk}

\begin{rmk}
	Suppose $A$ is a unital $C^*$-algebra, $X$ is a second countable compact Hausdorff space endowed with a free action of $G$, and $\alpha \colon G \to \Aut(A)$ is an action which has the $X$-Rokhlin property. If $B$ is any other $C^*$-algebra, then the action $\alpha \otimes \id_B$ of $G$ on $A \otimes_{\max} B$ has the $X$-Rokhlin property as well. 
	
	Furthermore, if $C$ is a finite dimensional $C^*$-algebra, then for any  finite $F \subseteq A$, for any finite $S \subseteq C \otimes C(X)$, and for any $\ep>0$, there exists a $G$-equivariant unital completely positive map $\varphi \colon C(X) \to A$ such that:
	\begin{enumerate}
		\item $\| \varphi(a)\varphi(b) - \varphi(ab) \|<\ep$ for all $a,b \in S$.
		\item $\| [ ( \id_C \otimes \varphi ) (a) , 1_C \otimes x ] \| < \ep$ for all $a \in S$ and for all $x \in F$.
	\end{enumerate}
We use those facts freely in the sequel.
\end{rmk}
We recall the definition of the join of two compact Hausdorff spaces.
\begin{dfn}
	 Let $X$ and $Y$ be compact Hausdorff spaces. We define an equivalence relation on $X \times Y \times [0,1]$ by setting $ (x,y,0) \sim (x',y,0)$ and $(x,y,1) \sim (x,y',1)$ for all $x,x' \in X$ and for all $y,y' \in Y$, and define
	\[
	X \star Y = ( X \times Y \times [0,1]  ) / { \sim } \,  .
	\]
If $X$ and $Y$ come equipped with left actions of a group $G$, then there is a naturally induced action of $G$  on the join: the action $g \cdot (x,y,t) = (g \cdot x , g \cdot y , t)$ on $X \times Y \times [0,1]$ descends to an action on the join. In particular, the action of $G$ on itself by left translation induces an action on the $k$-fold join $G^{\star k} = G \star G \star \cdots \star G$
(with $k$ copies of~$G$) for any $k \in \N$.
\end{dfn}

We recall the following explicit identification of the universal space for 
actions with finite Rokhlin dimension with commuting towers; see for example 
\cite[Lemma 4.3]{szabo-model-actions} for details.
\begin{lem}\label{L_2601_XRP}
Let $G$ be a compact group and let $A$ be a unital $C^*$-algebra.
An action $\af \colon G \to \Aut (A)$
satisfies $\RD (\af) \leq d$ \ifo{} $\alpha$ has the $X$-Rokhlin property
with $X$ being the $(d+1)$-fold join $G^{\star (d + 1)} = G \star G \star \cdots \star G$.
\end{lem}

\begin{prp}\label{P_2601_Ann}
Let $G$ be a compact group, let $A$ be a \ca,
and let $\af$ be an action of $G$ on~$A$.
Let $X$ be a compact free $G$-space and let $n \in \Nz$.
If $\af$ has the $X$-Rokhlin property
and $I (G)^n K_G^0 (X) = 0$,
then $I (G)^n K^G_* (A) = 0$.
\end{prp}

\begin{proof}
We first note that it is enough to show that the hypotheses imply $I (G)^n K^G_0(A) = 0$. To see this, notice that if $\af$ has the $X$-Rokhlin property, then the action $\af \otimes \id$ of $G$ on $A \otimes C(S^1)$ has the $X$-Rokhlin property as well, and $K^G_0 (A \otimes C(S^1) ) \cong K^G_0 (A) \oplus K^G_1 (A)$, so we can replace $A$ with $A \otimes C(S^1)$ for the purposes of the argument.

Let $H_0$ be a finite dimensional Hilbert space endowed with a unitary 
representation of $G$. Let $p \in (B(H_0) \otimes A)^G$ be a projection, and 
let $\eta \in I(G)^n$. We want to show that $\eta \cdot [p] = 0$. In order to 
simplify notation, we can replace $B(H_0) \otimes A$ with $A$ for the purpose 
of this argument, so we may assume without loss of generality that $p \in A$.

Write $\eta = [V_1] - [V_2]$, where $V_1$ and $V_2$ are finite dimensional unitary representations of $G$. Represent $[V_1]$ and $[V_2]$  by $G$-invariant projections $e_1,e_2$ in a finite dimensional representation space $W$. (For example, $W$ could be $V_1 \oplus V_2$ with the direct sum action.) 

Since $\eta \cdot [1_{C(X)}] = 0$ in $K^G_0(C(X))$, we have $[e_1 \otimes 1_{C(X)}] = [e_2 \otimes 1_{C(X)}]$ in $K^G_0(C(X))$. Therefore, possibly after adding trivial invariant projections, we can assume that $e_1 \otimes 1$ is Murray--von Neumann equivalent to $e_2 \otimes 1$ in $( B(W) \otimes C(X) )^G$. Pick $v \in ( B(W) \otimes C(X) )^G$ such that $v^*v = e_1 \otimes 1$ and $vv^* = e_2 \otimes 1$.

Choose a unital completely positive equivariant map $\varphi \colon C(X) \to A$ such that the following hold:
\begin{enumerate}
	\item $\| ( \id \otimes \varphi ) (v^*) \cdot ( \id \otimes \varphi ) (v) - ( \id \otimes \varphi ) ( e_1 \otimes 1) \| < 1/4$.
	\item $\| ( \id \otimes \varphi ) (v) \cdot ( \id \otimes \varphi ) (v^*) -  ( \id \otimes \varphi ) ( e_2 \otimes 1) \| < 1/4$.
	\item $\| \left [  ( \id_{B(W)} \otimes \varphi ) (v)   , 1_{B(W)} \otimes p  \right ] \| < 1/4$.
\end{enumerate}
Define 
\[
w =  ( \id_{B(W)} \otimes \varphi ) (v) \left ( 1_{B(W)} \otimes p \right ) 
\, .
\]
 Because $\varphi$ is equivariant, $w$ is fixed under the action of $G$. Also,
\begin{align*}
w^*w &=  \left ( 1_{B(W)} \otimes p \right ) ( \id_{B(W)} \otimes \varphi ) (v)^*   ( \id_{B(W)} \otimes \varphi ) (v) 
 \left ( 1_{B(W)} \otimes p \right ) 
 \\
 &\approx_{1/4} 
  \left ( 1_{B(W)} \otimes p \right )  ( \id_{B(W)}  \otimes \varphi )  ( e_1 \otimes 1_{C(X)})  \left ( 1_{B(W)} \otimes p \right ) \\
  & = e_1 \otimes p
\end{align*}
and
\begin{align*}
	ww^* &=  ( \id_{B(W)} \otimes \varphi ) (v)    \left ( 1_{B(W)} \otimes p \right )( \id_{B(W)} \otimes \varphi ) (v)^*
	\\
	&\approx_{1/4} 
 ( \id_{B(W)} \otimes \varphi ) (v)  ( \id_{B(W)} \otimes \varphi ) (v)^*  \left ( 1_{B(W)} \otimes p \right ) \\
	&\approx_{1/4}
	   ( \id_{B(W)}  \otimes \varphi ) ( e_2 \otimes 1_{C(X)})  \left ( 1_{B(W)} \otimes p \right ) 
	\\
	& = e_2 \otimes p \, .
\end{align*}
We thus have $\|w^*w - e_1 \otimes p\|<1/4$ and $\|ww^* - e_2 \otimes p\|<1/2$, so $[e_1 \otimes p] = [e_2 \otimes p]$, as required.
\end{proof}

The following proposition appears 
within the discussion on page 3 of \cite{AtSgl}.
\begin{prp}\label{prp:annihiliating-join}
	Let $G$ be a compact Lie group and let $d \in \N$. Then $I (G)^d K^*_G 
	(G^{\star d}) = 0$.
\end{prp}

Combining Propositions \ref{P_2601_Ann} and \ref{prp:annihiliating-join}, we 
obtain the following corollary.
\begin{cor}\label{C_2601_Lower}
	Let $\alpha$ be an action of a compact Lie group $G$ on a unital separable $C^*$-algebra $A$. 
	If $I (G)^n K^G_* (A) \neq 0$,
	then $\RD (\af) \geq n$.
\end{cor}

We will make use of the following fact, which can be extracted from the main 
results of \cite{AtSgl}. We refer the reader to the discussion in Section 2 of 
\cite{AtSgl}. The construction in \cite{AtSgl} makes use of Milnor's 
construction of the total space of the classifying space of $G$ as an inductive 
limit of repeated joins; however, it is explained in the end of the section that 
one could have used any other construction. 

\begin{prp}\label{P_2601_EGv2}
	Let $G$ be a compact Lie group and let $d \in \N$. Let $X_1 \subseteq X_2 
	\subseteq X_3 \subseteq \cdots$ be a sequence of inclusions finite CW complexes 
	equipped with free $G$ actions, such that the inductive limit is a model 
	for the total space $EG$ of the classifying space (namely, all homotopy 
	groups vanish). Then there exists $N_0$ such that for all $k \geq N_0$ 
	we have $I (G)^d K^*_G (X_k) \neq 0$.
\end{prp}
\begin{proof}
	Let $J_k$ be the kernel of the map $K_G^*(\{\textrm{pt}\}) \cong R(G) \to 
	K_G^*(X_k)$. Applying  \cite[Corollary 2.3]{AtSgl} and the discussion above to the case in which 
	the space $X$ there is a point, we see that the sequence of ideals $(J_k)_{k=1,2,\ldots}$ 
	defines the $I(G)$-adic topology on $R(G)$. In particular, for any $d$ 
	there exists $N_0$ such that for all $k \geq N_0$, the ideal $I(G)^d$ is 
	not contained in $J_k$, which is what is claimed. 
\end{proof}

 Lemma 1.9 of \cite{HP15} is stated for finite group actions. We need a 
 generalization for actions of compact Lie groups. The idea of proof is 
 similar; the required modification in the proof involves using the free action 
 case of the Slice Theorem for compact Lie group actions (\cite[Section 
 3]{Mostow}) in place of the use of local liftings in for finite group actions 
 in the proof of \cite[Lemma 1.9]{HP15}.  We omit the details; we refer the 
 reader to \cite[Theorem 5.16]{GHTW} for a more general discussion.
\begin{lem}
	\label{lem:X-implies-dimrok}
	Let $A$ be a unital separable $C^*$-algebra, let $G$ be a compact Lie group, and let 
	$\alpha \colon G \to \Aut(A)$ be an action. Let $X$ be a compact free $G$-space such that the covering dimension of $X/G$ is $d$. If $\alpha$ has the $X$-Rokhlin property then $\RD(\alpha) \leq d$.
\end{lem}

\section{Constructions}\label{Sec_Exs}

In order to construct examples, we need the following simple fact concerning $K$-theory of joins of finite sets. This is a straightforward exercise in algebraic topology. Since we did not find a reference for it, we include a proof for the reader's convenience.

\begin{lem}
	Let $S$ be a finite set of cardinality $N$. Let $l,r \in \N \cup \{0\}$, 
	and let $X$ be a compact Hausdorff space such that $K^0(X) \cong \Z^l$ and 
	$K^1  (X) \cong \Z^r$. Then 
	\[
	K^0(X \star S) \cong \Z^{r (N-1)  +1} \qquad \mbox{and} \qquad 
	K^1(X \star S) \cong \Z^{ (l-1) (N-1) } \,
	.
	\]
\end{lem}
\begin{proof}
	Set 
	\[
	Y =  ( X \times S \times [0,2/3] ) / {\sim} \qquad \mbox{and} \qquad  Z = ( X \times S \times [1/3,1] ) / { \sim }
	\, .
	\]
	 Then $Y$ is homotopy equivalent to $X$, $Z$ is homotopy equivalent to $S$, and $Y \cap Z$ is homotopy equivalent to $X \times S$. Consider the Mayer-Vietoris sequence
	\[
 \xymatrix{
 	K^0(X \star S) \ar[r]^-{\textrm{res}_0} & K^0(Y) \oplus K^0(Z) \ar[r]^-{\Delta_0} & K^0(Y \cap Z) \ar[d] \\
 	 K^1(Y \cap Z) \ar[u] & \ar[l]^-{\Delta_1}  K^1(Y) \oplus K^1(Z)  & \ar[l]^-{\textrm{res}_1}  K^1(X \star S) \, .
	} 
	\]
	Filling in the groups which are given in the hypothesis, we obtain
		\[
	\xymatrix{
		K^0(X \star S) \ar[r]^-{\textrm{res}_0} & \Z^l \oplus \Z^N \ar[r]^-{\Delta_0} & \Z^l \otimes \Z^N \ar[d] \\
		\Z^r \otimes \Z^N \ar[u] & \ar[l]^-{\Delta_1}  \Z^r \oplus 0 & \ar[l]^-{\textrm{res}_1}  K^1(X \star S) \, .
	}
	\]
	Identifying $\Z^l \cong M_{l,1}(\Z)$ (integer column vectors), $\Z^N \cong 
	M_{1,N}(\Z)$ (integer row vectors), and  $\Z^l \otimes \Z^N \cong 
	M_{l,N}(\Z)$, we can write $\Delta_0$ in the first row as 
	\[
	\Delta_0(a,b) = a \cdot ( \begin{matrix} 1 & 1 & \cdots & 1 \end{matrix} ) - \left ( \begin{matrix} 1 \\ 1 \\ \vdots \\ 1 \end{matrix} \right ) \cdot b 
	\, .
	\]
	One readily verifies that the image of $\Delta_0$ is generated by elements which are not divisible by any integer greater than $1$, and its kernel is generated by the single element $(a_0,b_0)$ with $a_0 =  \left ( \begin{matrix} 1 \\ 1 \\ \vdots \\ 1 \end{matrix} \right )$ and $b_0 = ( \begin{matrix} 1 & 1 & \cdots & 1 \end{matrix} ) $. 
	Thus $\ker(\Delta_0) \cong \Z$ and $\mathrm{coker}(\Delta_0) \cong 
	\Z^{l\cdot N - l - N +1}$. 
	Furthermore, the map $\Delta_1$ is injective 
	and its image is generated by elements which are not divisible by any 
	integer greater than $1$. Since all the groups involved are free abelian, 
	and thus the resulting short exact sequences split, we can put this all 
	together and obtain
	\[
	K^1(X \star S) \cong \mathrm{coker}(\Delta_0) \cong Z^{l\cdot N - l - N +1}
	\]
	and 
	\[
	K^0(X \star S) \cong \mathrm{coker}(\Delta_1) \oplus \ker(\Delta_0) \cong Z^{r\cdot N - r} \oplus \Z 
	\, ,
	\]
	as required.
\end{proof}

\begin{cor}
	Let $N \in \{ 2,3,\ldots\}$, and let $S$ be a finite set of cardinality $N$. Then for any $k \in \N$, we have
	\[
	K^0(S^{\star k}) \cong \left \{  \begin{matrix} 
			\Z^{(N-1)^k + 1} & \colon & k \; \mathrm{ odd } \\
			\Z & \colon & k \; \mathrm{ even } \\
		\end{matrix}   \right . 
	\]
	and
	\[
	K^1(S^{\star k}) \cong \left \{  \begin{matrix} 
		0 & \colon & k \;  \mathrm{ odd } \\
		\Z^{ (N-1)^k } & \colon & k \;  \mathrm{ even }  \, . \\
	\end{matrix}   \right . 
	\]
\end{cor}

\begin{thm}\label{thm:high-rd}
	Let $G$ be a finite group. For any $n \in \N$ there exists a simple unital AF algebra with unique trace along with an action $\alpha \colon G \to \Aut(A)$ such that $n <  \RD(\alpha) < \infty$. 
\end{thm}
\begin{proof}
	Fix $n \in \N$. Using Proposition \ref{P_2601_EGv2}, and recalling that one can take $X_k = G^{\star k}$ there, fix an odd number $k > 
	n$ such that $I(G)^n K_G^*(G^{\star k}) \neq 0$. Define $X = G^{\star k}$.  
	Pick a dense sequence $x_1,x_2,\ldots \in X$. Let $N$ be the cardinality of 
	$G$. For any $m > 0$, define $\varphi_m \colon C(X) \to M_{N+1} \otimes 
	C(X)$ as follows. Index the rows and columns of $M_{N+1}$ by $\{0\} \cup 
	G$. For $g,h \in \{0\} \cup G$, let $e_{g,h}$ be the corresponding standard matrix unit in $M_{N+1}$. Set 
	\[
	\varphi_m (f) = e_{0,0} \otimes f \oplus \bigoplus_{g \in G} \left [ e_{g,g} \otimes f(g \cdot x_m) \cdot 1_{C(X)}  \right ] \; .
	\] 
	For $g \in G$, we define a representation $g \mapsto u_g \in M_{N+1}$ to be the 
	direct sum of the trivial representation and the left regular 
	representation, given as follows: If we denote by $( \delta_{x} )_{x \in \{0\} \cup 
	G}$ the orthonormal basis of $\C^{N+1}$ corresponding to the indexing of 
	$M_{N+1}$  above, then $u_g \delta_0 = \delta_0$ and, for $h \in G$, we 
	have $u_g \delta_h = \delta_{gh}$. 
	
	We set $A_0 = C(X)$, and inductively define $A_{m+1} = M_{N+1} \otimes A_m$. For $m \in \Nz$, define $\psi_{m+1,m} \colon A_{m} \to A_{m+1}$ by $\psi_{m+1,m} = \id_{M_{N+1}^{\otimes m}} \otimes \varphi_m$. 
	Define $A = \underset{\longrightarrow}{\lim} (A_m , \psi_{m+1,m})$. Set 
	$\psi_{\infty,m} \colon A_m \to A$ to be the canonical map of $A_m$ into 
	the inductive limit. 
	
	For $g \in G$ and $m \in \Nz$, we define $\alpha^{(m)}_g \in \Aut(A_m)$ by 
	\[
	\alpha^{(m)}_g(f) (x) = \Ad (u_g^{\otimes m}) \otimes f (g^{-1} \cdot x)  
	\; .
	\]
	 For all $m \in \Nz$ we have $\alpha^{(m+1)}_g \circ \psi_{m+1,m} =\psi_{m+1,m} \circ \alpha^{(m)}_g$. Thus, the automorphisms $ \alpha^{(m)}_g $ define an inductive limit action $g \mapsto \alpha_g \in  \Aut(A)$. 
	 
	 We claim that $A$ is a simple AF algebra with unique trace and that $k \geq \RD(\alpha) > n$. 
	 
	 The construction of $A$ follows the method discussed in \cite{Goodearl}. 
	 That $A$ is simple is a standard argument which follows from the 
	 assumption that the sequence $x_1,x_2,\ldots$ is dense and that $X$ is 
	 connected: if $a \in A_m \smallsetminus \{0\}$ for some $m \in \Nz$ then 
	 there exists $l \geq m$ such that the image of $a$ in $A_l$ is full (that 
	 is, does not vanish at any point); see \cite[Lemma 1]{Goodearl}. It 
	 follows from \cite[Theorem 13]{Goodearl} and the following discussion that 
	 $K_0(A)$ satisfies the Riesz interpolation property. Furthermore, since 
	 $K_0(A_m)$ is torsion free for all $m$, so is $K_0(A)$. It is also 
	 immediate that $K_0(A) \not \cong \Z$. Because $K_1(A_m) = 0$ for all $m$, 
	 we have $K_1(A) = 0$. We now claim that $A$ has a unique trace. Though 
	 this follows from known techniques in the literature, we did not find a 
	 convenient reference; we thus provide a full argument for the 
	 convenience of the reader. 
	 We denote by $\tr$ the unique (normalized) tracial state on $M_k$, for any $k$; we omit the subscript $k$ as there is no risk of confusion.	 
	 Let $\tau$ and $\tau'$ be two tracial states on $A$. It suffices to show that $\tau(f) = \tau'(f)$ for any $f$ from one of the steps of the inductive limit. Set $\tau_m = \tau|_{A_m}$ and   $\tau'_m = \tau'|_{A_m}$. These sequences of traces satisfy 
	 \[
	 \tau_{m+1} \circ \psi_{m+1,m} = \tau_m \quad \mbox{and} \quad \tau'_{m+1} \circ \psi_{m+1,m} = \tau'_m \quad \mbox{for all } m \in \N \, .
	 \]
	    For $l>m$,  we set
	 \[
	 \psi_{l,m} = \psi_{l-1,l-2} \circ \psi_{l-2,l-3} \circ \cdots \circ \psi_{m+1,m} \colon A_m \to A_l
	 \, ,
	 \]
	  with $\psi_{m,m} = \id_{A_m}$.
	 
	 Fix $m_0 \in \N$ and $f \in A_{m_0}$. 
	 Then 
	 \[
	 ( \tau_{m+1} \circ \psi_{m+1,m} ) (f) = \frac{1}{N+1} \left ( \tau_m(f) + \sum_{g \in G} \tr(f(g \cdot x_m ) )   \right ) 
	 \]
	 and
	 \[
	 ( \tau'_{m+1} \circ \psi_{m+1,m} ) (f) = \frac{1}{N+1} \left ( \tau'_m(f) + \sum_{g \in G} \tr(f(g \cdot x_m ) )   \right )  \, .
	 \]
	  Therefore
	 \[
	  \| \tau_{m+1} \circ \psi_{m+1,m} - \tau'_{m+1} \circ \psi_{m+1,m} \| 
	  \leq \frac{1}{N+1} \| \tau_{m} - \tau'_{m} \| \, .
	 \]
	 By repeating the argument, we see that for any $j \in \N$, we have
	 \[
	 \| \tau_{m+j} \circ \psi_{m+j,m} - \tau'_{m+j} \circ \psi_{m+j,m} \| 
	 \leq \frac{1}{ ( N+1 )^j } \| \tau_{m} - \tau'_{m} \| \, .
	 \]
	 Consequently, $\tau = \tau'$. 
	 
	 So far, we have shown that $A$ is a simple AH algebra with no dimension 
	 growth, and that it has the same Elliott invariant as a simple AF algebra. 
	 By the classification theorem for $AH$ algebras with no dimension growth 
	 (\cite{EGL}), it follows that $A$ is in fact a simple AF algebra.
	 
	 Now define $\sigma_m \colon C(X) \to A_m$ by $\sigma_m(f) = 
	 1_{M_{N+1}^{\otimes m}} \otimes f$. Each such map is 
	 $G$-equivariant, and its image is in the center of $A_m$. 
	 Therefore $ ( \psi_{\infty,m} \circ \sigma_m )_{m=1,2,\ldots}$ is a 
	 central sequence of $G$-equivariant unital homomorphisms from $C(X)$ to 
	 $A$. It follows that $\RD(\alpha) \leq k$.
	 
	 In order to show that $\RD(\alpha) \geq n$, by Corollary 
	 \ref{C_2601_Lower} it suffices to show that 
	 $I(G)^nK^G_*(A) \neq 0$.
	  Pick an element $\eta \in K^G_*(X)$ such that 
	 $I(G)^n \eta \neq 0$. It suffices to show that for any $m \in \N$ we have 
	 $I(G)^n ( \psi_{m,1})_* (\eta) \neq 0$. 
	 If $\eta \in K_G^1(X)$ then this 
	 is immediate, because $( \psi_{m,1})_* $ is an isomorphism. If $\eta \in 
	 K_G^0(X)$ then $( \psi_{m,1})_*  (\eta)$ is the direct sum of $\eta$ with an 
	 integer multiple of the class of the left regular representation 
	 (depending on $m$ and the rank of $\eta$). The copies of the class of the 
	 left regular representation are annihilated by the augmentation ideal, and 
	 therefore we have $I(G)^n ( \psi_{m,1})_*  (\eta) \neq 0$, as required.  
\end{proof}

In the specific case in which $G = \Z_2$, we can obtain a somewhat improved result.

\begin{prp}\label{I_2601_Z2}
There exist a simple unital AF algebra $A$ with unique trace, and 
actions $\alpha_m \colon \Z_2 \to \Aut(A)$ for $m \in \N$, such that  for all $m$, we have $m \leq
\RD(\alpha_m) \leq 2m+2$.
\end{prp}

\begin{proof}
	We repeat the construction in Theorem \ref{thm:high-rd}. Recall that 
	$\Z_2^{\star k} \cong S^{k-1}$. This fact is well known; we did not find a good reference, but it is easy to prove. Thus, regardless of which odd $k$ we 
	choose, the resulting AF-algebra has the same $K$-theory, independently 
	of the choice of $k$. Therefore, we obtain actions on the same AF algebra. 
	As in 
	the proof of Theorem \ref{thm:high-rd}, if we pick an odd number $k > m$ 
	such that $I(\Z_2)^m K_{\Z_2}( \Z_2^{\star k} ) \neq 0$ then we have $k 
	> \RD(\alpha) \geq m$. It remains to show that we can choose $k \leq 2m+3$. 

	Using the computations on pages 103--107 of~\cite{Atyh}, we see that for any $l$, we have
	\[
	K_{\Z_2}^0 ( \Z_2^{\star (2 l + 1)} ) \cong K^0 (\R P^{2 l})
	\cong R (\Z_2) / I (\Z_2)^l.
	\]
Thus, for $k=2m+3$ we have 
\[
I(\Z_2)^m K_{\Z_2}^0(\Z_2^{\star k} ) \cong I (\Z_2)^m / I (\Z_2)^{m+1} \neq 0
\, ,
\]
as required.	
\end{proof}

We now turn to the case of non-discrete compact Lie groups. For the case of 
actions of the circle group, we obtain a sharp result.

\begin{prp}\label{I_2601_S1}
For every $d \in \Nz$ there is a simple unital AH~algebra~$A$
with no dimension growth and an action $\af \colon S^1 \to \Aut (A)$
such that $\RD (\af) = d$.
\end{prp}

\begin{proof}
Consider the direct system
\[
C (S^{2 d + 1})
 \longrightarrow C (S^{2 d + 1}, M_2)
 \longrightarrow C (S^{2 d + 1}, M_4)
 \longrightarrow \cdots.
\]
The actions are all induced from the scalar multiplication action
of $S^1$ on $S^{2 d + 1} \subseteq \C^{d + 1}$. 
The maps are all diagonal, of the form $f \mapsto \diag(f,f \circ T_n)$, where $T_n \colon S^{2 d + 1} \to S^{2 d + 1}$ is given by $T_n (\xi) = u_n \xi$ for a unitary $u_n \in U(\C^{d + 1})$. 
If the unitaries are chosen to be dense in $U(\C^{d + 1})$,
the direct limit will be simple.

The upper bound for $\RD (\af)$ comes from Lemma~\ref{L_2601_XRP}
and the observation (stated but not proved in the beginning
of the proof of \cite[Lemma~3.1]{AtSgl})
that $(S^1)^{\star n} \cong S^{2 n - 1}$ equivariantly.

We use Corollary~\ref{C_2601_Lower} to prove the lower bound.
First, it follows from \cite[Lemma 3.1]{AtSgl} that for any $n\in \N$, we have 
$ K^*_{S^1}((S^1)^{\star n} ) \cong R(S^1)/I(S^1)^n   $ as $R(S^1)$-modules. 
Therefore, whenever $n \leq d$, we have $I(S^1)^n K^*_{S^1}((S^1)^{\star ( d+1 
) }  )\neq 0$. 
Since the action is free, $K^0_{S^1}((S^1)^{\star ( d+1 ) } )$, as an abelian group, is isomorphic to 
$K^0(S^{2d+1}/S^1) = K^0(\C \mathbb{P}^{d+1}) \cong \Z^{d+2}$. 
The connecting maps in the direct system induce multiplication by $2$ on $K^0$. Now, for any $n \leq d$  and for any $m \in \N$, we more specifically have 
$I(S^1)^n \cdot  2^m [1]_{ K^0_{S^1}((S^1)^{\star (d+1)} } \neq 0$. 
Thus, $I(S^1)^n [1]_{K_0^{S^1}(A)} \neq 0$, so $I(S^1)^nK_0^{S^1}(A)\neq 0$,
 and therefore $\RD(\alpha) > n$, as claimed.
\end{proof}

\begin{qst}\label{Q_2531_S1_res}
	What happens to $\RD (\cdot )$ if one restricts the actions
	in the proof above to finite (cyclic) subgroups of~$S^1$?
\end{qst}

We now prove a version of Theorem \ref{thm:high-rd} which applies to arbitrary 
Lie groups. The proofs of Theorem \ref{thm:high-rd} as well as Proposition 
\ref{I_2601_S1} involve the action of $G$ on repeated joins, whose inductive 
limit is the model for the total space $EG$ of the classifying space for $G$. 
However, it is not clear how to use this model  to construct a simple 
$C^*$-algebra in this more general context. In Theorem \ref{thm:high-rd} this is 
done via point evaluations, 
which cannot be done in an equivariant way if the group is not discrete. In 
Proposition \ref{I_2601_S1} this is done using a transitive group action on 
the repeated join which commutes with the $S^1$ action, which can be done as we 
have an explicit identification of those repeated joins as higher dimensional 
spheres, but, to our knowledge, this cannot be done for general Lie groups. 
Instead, 
we use a different model for approximations for $EG$.
\begin{lem}\label{lem:transitive-action}
	Let $G$ be a connected compact Lie group, and let $M$ be a connected compact smooth 
	manifold endowed with a free action of $G$. Let $H < \textrm{Homeo}(M)$ be 
	the group of homeomorphisms which commute with the action of $G$ and are 
	$G$-equivariantly null-homotopic. Then $H$ 
	acts transitively on $M$. 
\end{lem}
\begin{proof}
	As $G$ is compact and the action is free, the quotient space $M/G$ is a 
	manifold. Denote by $D$ the closed unit ball in the Euclidean space 
	of dimension $\dim(M/G)$. Let $x \in M$. Using the Slice Theorem 
	(\cite[Section 3]{Mostow}), we can choose a $G$-invariant closed tubular 
	neighborhood $N \cong G \times D$ of $Gx$ 
	in $M$, where the identification with $G \times D$ is 
	such that the action on $G$ is by left translation and the action on $D$ is 
	trivial. We claim that it suffices to show that we can find homeomorphisms of $G \times 
	D$ which commute with the action of $G$, fix the boundary $G \times 
	\partial D$, and are transitive on the interior $G \times \mathrm{int}(D)$. 
	To see this, first, such homeomorphisms can be extended to homeomorphisms of $M$ which fix all 
	points outside of $N$. It follows that the orbits of the action of $H$ on $M$ are open. Because $M$ is connected, compositions of such homeomorphisms can be used to 
	obtain a transitive action on $M$, as claimed.
	
	We can construct 
	homeomorphisms of $G \times D$ which are transitive and commute with the 
	action of 
	$G$ as compositions of the following two types of homeomorphisms. For the 
	first type, suppose $h \colon D \to D$ is a homeomorphism which fixes the 
	boundary. Then we can define a homeomorphism of $G \times D$ by $T(g,x) = 
	(g,h(x))$. For the second kind, let $g_0 \in G$, and let $( g_t )_{t \in 
	[0,1]}$ be a path in $G$, with $g_0$ as given, such that $g_1 = 1$. Recall 
	that $D$ is the closed unit 
	ball in a Euclidean space, and define $T (g,x) = ( g \cdot g_{\|x\|} , x )$. 
	One readily checks that one can translate any point in the interior of $G 
	\times D$ 
	to 
	any other via a composition of such maps. 
\end{proof}

\begin{prp}\label{I_2601_ArbLgFin}
	Let $G$ be a compact connected Lie group.
	Then for every $d \in \Nz$ there is a simple unital AH~algebra~$A$ with
	no dimension growth and an action $\af \colon G \to \Aut (A)$
	such that $d \leq \RD (\af) < \I$.
\end{prp}

\begin{proof}
	Fix $n$ such that $G$ embeds into the unitary group $U_n$. For any $m \geq 
	n$, define the space $C_{n,m}$ of orthonormal $n$-frames in $\C^m$:
	\[
	C_{n,m} = \left \{ (\xi_1,\xi_2,\ldots,\xi_n) \in ( \C^m )^n \mid 
	\left < \xi_j , \xi_k \right > = \delta_{j,k} \mbox{ for all } j,k \in \{1,2,\ldots,n\} \right  \} \, .
	\]
	Define an action of $U_n$ on $C_{n,m}$ as follows. For  a unitary $n\times n$ matrix
	$u = 
	( u_{j,k})_{j,k \in \{1,2,\ldots n \} }$, set
	\[
	u \cdot (\xi_1,\xi_2,\ldots,\xi_n) = 
		\left ( 
	\sum_{k=1}^n u_{1 k} \xi_k, 
	\sum_{k=1}^n u_{2 k} \xi_k ,\ldots, \sum_{k=1}^n u_{n k} \xi_k 
	\right )
	\, . 
	\]
	This action is free and continuous.
	
	Any isometric inclusion $\C^m \to \C^{m+1}$ induces an equivariant embedding $C_{n,m} \to 
	C_{n,m+1}$, and all such embeddings are homotopy equivalent.  We can thus define the 
	inductive limit space $C_n = 
	\underset{\longrightarrow}{\lim} \, C_{n,m}$ with the inductive limit 
	topology, 
	which carries a free action of $G$. We claim that $C_n$ is a model for $EG$. 
	To see this, we need to show that all homotopy groups of this space vanish.
	To that end, it suffices to show that for any $j \in \N$ there exists 
	$m_{j,n}$ such that for any $m \geq m_{j,n}$ we have $\pi_j (C_{n,m}) = 0$. 
	
	Note that $C_{1,m} \cong S^{2m-1}$. For any $m \geq n$, we have a 
	surjective map $C_{n,m} \to  C_{1,m}$, given by 
	$(\xi_1,\xi_2,\ldots,\xi_n) \mapsto \xi_1$. The fiber over a point $\xi$ 
	consists of all frames of $n-1$ vectors in $\{\xi\}^{\perp} \cong 
	\C^{m-1}$, so each fiber is homeomorphic to $C_{n-1,m-1}$. We can now prove 
	the claim by induction on $n$.  For $n=1$, we can pick $m_{j,1} = \lceil j/2 
	\rceil +1$ (where $\lceil j/2 
	\rceil$ is the smallest integer greater or equal to $j/2$).
	 Now, 
	assume by induction that for all $k<n$ there exist numbers $m_{j,k}$  such that for any $m \geq m_{j,k}$ we have $\pi_j (C_{k,m}) = 0$. Suppose 
	 $m \geq \max ( \{ m_{j,n-1} + 1 , m_{j,1} \} )$. By the inductive hypothesis, we have $\pi_j(C_{n-1,m-1}) = 0$, and  also  $\pi_j(S^{2 m-1}) = 0$. Thus,
	    by the long exact 
	sequence for the homotopy of a fibration (\cite[Theorem 4.41]{hatcher}), we 
	also 
	have $\pi_j(C_{n,m}) = 0$, as required.
	
	By Proposition \ref{P_2601_EGv2}, for any $n$ and for any $d$ there exists 
	$m$ such that $I (G)^d K^*_G (C_{n,m}) \neq 0$. Fix $d$ and fix such an 
	integer $m$. Set $X = C_{n,m}$. 
	Let $H$ be the 
	group of homeomorphisms of $X$ which commute with the action of $U_n$ and 
	are $U_n$-equivariantly null-homotopic.  By Lemma \ref{lem:transitive-action}, $H$ is transitive. We can thus 
	choose homeomorphisms $h_k \colon X \to X$  which commute with the 
	action of $G$ and such that for any $x \in X$, the set $\{h_k(x) \mid k \in \N 
	\}$ 
	is dense in $X$. Pick orientation reversing  
	homeomorphisms $\theta_k \colon S^1 \to S^1$ such that for any $x \in S^1$, $\{\theta_k(x) \mid k \in \N 
	\}$ is dense in $S^1$ . Define $\tilde{h}_k \colon X \times S^1 \to X \times S^1$ by 
	$\tilde{h}_k (x,z) = (h_k(x),z)$ and $\tilde{\theta}_k \colon X \times S^1 
	\to X \times S^1$ by $\tilde{\theta}_k(x,z) = (x,\theta_k(z))$.
	
	Consider the direct system
	\[
	C (X \times S^1)
	\longrightarrow C (X \times S^1, M_3)
	\longrightarrow C (X\times S^1, M_9)
	\longrightarrow \cdots.
	\]
	with connecting homomorphisms 
	
	\[\varphi_k \colon C(X \times S^1,M_{3^{k-1}}) 
	\to  
	C(X \times S^1,M_{3^{k}})
	\]
	 given by $\varphi_k(f) = \diag (f,f \circ 
	\tilde{h}_k,f \circ \tilde{\theta}_k)$. Denote the inductive limit by $A$. By the 
	choice of the sequences $( h_k )_{ k \in \N }$  and  $( \theta_k )_{ k \in \N }$, one can check that  if $f$ is a non-zero element coming from an algebra in one of the steps, then its image in the algebra at large enough step must be full; thus, the 
	inductive limit is simple. Because  
	$h_k$ commutes with the action of $G$ for each $k$, we have consistent continuous 
	actions of $G$ on the algebras in the inductive system, and thus an action on the direct
	limit. Because we have a central sequence of $G$-equivariant embeddings of 
	$C(X)$ into the inductive limit, by Lemma~\ref{lem:X-implies-dimrok}, the action $\alpha$ on the inductive limit 
	has 
	Rokhlin dimension at most $\dim(X/G)$, which is finite. 
	
	By our construction, the connecting maps induce isomorphisms on $ { K^1_G (C(X 
	\times S^1) ) } $. Thus, the embedding  $C(X\times 
	S^1) \to A$ from the first level induces an isomorphism of $G$-equivariant $K^1$ groups. 
	Furthermore, whenever $j \leq d$, we have $I(G)^j K^*_{G}(X) \neq 0$. Thus, 
	whenever $j \leq d$, we have $I(G)^j K^1_{G}(X \times S^1) \neq 0$, and 
	therefore $I(G)^j K^1_{G}(A) \neq 0$. 	So, by 
	Corollary~\ref{C_2601_Lower}, we have $\RD(\alpha) \geq d$. 
\end{proof}

\section{Tensor product actions} \label{sec:tensor-product-actions}
In this section, we obtain some bounds and examples involving the following two 
related situations. One situation involves a compact Lie group with actions on 
two $C^*$-algebras, $\beta^{(1)} \colon G \to \Aut(A_1)$ and $\beta^{(2)} 
\colon G \to \Aut(A_2)$, and we consider the action $\beta^{(1)} \otimes 
\beta^{(2)}$ of $G$ on the tensor product $A_1 \otimes_{\max} A_2$. The second 
case involves two compact Lie groups $H_1$ and $H_2$ with actions $\beta^{(1)} 
\colon H_1 \to \Aut(B_1)$ and $\beta^{(2)} \colon H_2 \to \Aut(B_2)$, and we 
consider the action $\beta^{(1)} \otimes \beta^{(2)}$ of $H_1 \times H_2$ on 
$B_1 
\otimes_{\max} B_2$. Those are used to construct actions of groups which have 
$S^1$ or $\Z_2$ as factors with Rokhlin dimension within a prescribed range, as 
well as examples of pairs of actions of $\Z_6$ on simple AF algebras which have 
arbitrarily large finite Rokhlin dimension, but whose tensor product action has 
the Rokhlin property.

\begin{prp}\label{P_2531_ExtTens}
For $j = 1, 2$ let $H_j$ be a compact Lie group, let $B_j$ be a separable \ca,
and let $\bt^{(j)}$ be an action of $H_j$ on~$B_j$.
Then the action
$\af \colon H_1 \times H_2 \to \Aut ( B_1 \otimes_{\max} B_2 )$
given by $\af_{h_1, h_2} = \bt^{(1)}_{h_1} \otimes \bt^{(2)}_{h_2}$ satisfies
\[
\RD (\af) 
  \leq \RD (\bt^{(1)} ) + \RD (\bt^{(2)} ) \, .
\]
\end{prp}

\begin{proof}
	For $j=1,2$, set $d_j = \RD (\bt^{(j)})$ and  $X_j = H_j^{\star 
	(d_j+1)}$. 
	
	Consider the homomorphism 
	\[
	\pi \colon l^{\infty}(B_1) \otimes_{\max} 
	l^{\infty}(B_2) \to l^{\infty}(B_1 \otimes_{\max} B_2)
	\]
	 defined on simple 
	tensors by $\pi(f_1 \otimes f_2) (n) = f_1(n) \otimes f_2(n)$. 
	 If 
	$f_j \in  l^{\infty , ( \bt^{(j)} ) } (B_j)$ for $j=1,2$ then 
	$\pi(f_1 
	\otimes f_2) \in l^{\infty , ( \af )} (B_1 \otimes_{\max} B_2)$.
	This map induces a homomorphism $( B_{1}) _{\infty}^{(\bt^{(1)})} 	\otimes_{\max} 
	( B_{2}) _{\infty}^{(\bt^{(2)})} \to (B_1 \otimes_{\max} B_2 
	)_{\infty}^{(\af)}$ which by slight abuse of 
	notation we  also denote by $\pi$. One checks that $\pi$ furthermore respects central 
	sequences and annihilators, and thus we obtain a unital homomorphism 
	$\pi \colon F_{\infty}^{(\beta^{(1)})}(B_1) \otimes_{\max} 
	F_{\infty}^{(\beta^{(2)})}(B_2) \to F_{\infty}^{(\alpha)}(B_1 
	\otimes_{\max} B_2)$. By assumption, there exist equivariant unital 
	homomorphisms $C(X_j) \to  F_{\infty}^{(\beta^{(j)})}(B_j)$, and we thus 
	have an $H_1 \times H_2$-equivariant unital homomorphism 
	\[
	C(X_1 \times X_2) 
	\cong C(X_1) \otimes C(X_2) \to  F_{\infty}^{(\alpha)}(B_1 
	\otimes_{\max} B_2)
	\, .
	\]
	 Now, 
	 \[
	 \dim( ( X_1 \times X_2 ) / ( H_1 \times H_2 ) ) = d_1 + 
	d_2
	\, ,
	\]
	 and $H_1 \times H_2$ acts freely on 
	$X_1 
	\times X_2$. Thus, $\af$ has the $(X_1 \times X_2 )$-Rokhlin property, and therefore, by Lemma~\ref{lem:X-implies-dimrok}, we have $\RD (\af)  \leq d_1 + d_2$, as claimed.
\end{proof}

\begin{cor}\label{C_2601_TnsR}
Under the hypotheses of Proposition~\ref{P_2531_ExtTens},
if $\bt^{(2)}$ has the Rokhlin property,
then $\RD (\af) \leq \RD (\bt^{(1)} )$.
\end{cor}

\begin{proof}
We have $\RD (\bt^{(2)} ) = 0$.
\end{proof}

\begin{exa}\label{exl:reduction-in-dimension}
We can have strict inequality 
in Proposition~\ref{P_2531_ExtTens}. Fix nontrivial finite groups $H_1$ and $H_2$. 
Using 
Theorem \ref{thm:high-rd}, find an AF algebra $B_1$ and an action $\bt^{(1)} 
\colon H_1 \to \Aut(B_1)$ such that $0 <  \RD (\bt^{(1)} ) < \infty$. Let $B_2 = 
\mathcal{O}_2$, and fix an action  $\bt^{(2)} 
\colon H_2 \to \Aut(B_2)$ with the Rokhlin property. The corresponding action 
of $\alpha$ of $H_1 \times H_2$ on $A = B_1 \otimes B_2 \cong \OT$ has finite 
Rokhlin dimension with commuting towers, and therefore it has the 
Rokhlin property by \cite[Proposition 4.32]{GdHbSg}.
\end{exa}

\begin{rmk}\label{rmk:reduction-in-dimension}
	Example \ref{exl:reduction-in-dimension} reduces the Rokhin dimension all 
	the way down to zero. It would be interesting to see if one can find 
	examples where there is strict inequality, but the Rokhlin dimension of the 
	product action is not zero.
\end{rmk}
 
The following result is for an action formally similar to that of
Proposition~\ref{P_2531_ExtTens}.

\begin{prp}\label{P_2531_DTens}
Let $G$ be a compact group.
For $j = 1, 2$ let $B_j$ be a separable \ca,
and let $\bt^{(j)}$ be an action of $G$ on~$B_j$.
Then the action
$\af \colon G \to \Aut ( B_1 \otimes_{\max} B_2 )$,
given by $\af_{g} = \bt^{(1)}_{g} \otimes \bt^{(2)}_{g}$, satisfies
\[
\RD (\af ) \leq \min ( \RD (\bt^{(1)}), \, \RD (\bt^{(2)} ) ).
\]
\end{prp}

\begin{proof}
Assume without loss of generality that the minimum is attained at $\bt^{(1)}$. 
As in Proposition \ref{P_2531_ExtTens}, we have  a $G$-equivariant unital 
homomorphism 	
\[ F_{\infty}^{(\beta^{(1)})}(B_1) \otimes_{\max} 
F_{\infty}^{(\beta^{(2)})}(B_2) \to F_{\infty}^{(\alpha)}(B_1 
\otimes_{\max} B_2)
\, .
\]
With $r = \RD (\bt^{(1)})+1$, let $X = G^{\star r}$. We have an equivariant unital homomorphism $C(X) 
	\to  F_{\infty}^{(\beta^{(1)})}(B_1) $, and therefore also equivariant 
	unital homomorphisms 
	$C(X) \to  F_{\infty}^{(\beta^{(1)})}(B_1) \otimes 1 \to  
	F_{\infty}^{(\alpha)}(B_1 
	\otimes_{\max} B_2)$. Thus, $ \RD (\af ) \leq  \RD (\bt^{(1)})$. 
\end{proof}

For the next proposition, we need the following K{\"u}nneth-type formula. Note 
that if $G$ and $H$ are compact Lie groups, then $R(G \times H) \cong R(G) 
\otimes_{\Z} R(H)$. (This can be seen from the fact that any irreducible 
representation of $G \times H$ decomposes as a tensor product of an irreducible 
representation of $G$ by with an irreducible representation of $H$; see 
\cite[Exercise~2.36]{fulton-harris}.) The K{\"u}nneth formula does not 
generalize to equivariant $K$-theory in general. However, we want to consider 
the following case: if $A$ and $B$ are $C^*$-algebras endowed with actions of 
$G$ and $H$, we want to understand the $R(G \times H)$-module structure of 
$K_*^{G \times H}(A \otimes  B)$. To avoid technical complications, we restrict 
ourselves to the case in which at least one of those algebras is nuclear, so 
there is no question as to which tensor product is used.   Notice that 
$K_*^G(A) \otimes_{\Z} K_*^H(B)$ is naturally endowed with an $R(G) 
\otimes_{\Z} R(H)$-module structure. Likewise,  $\text{Tor}(K_*^G(A) , K_*^H(B) 
)$ is endowed with an $R(G) \otimes_{\Z} R(H)$-module structure, as can be seen 
by its construction as a homology group of an exact sequence of  $R(G) 
\otimes_{\Z} R(H)$-modules. We show here that the K{\"u}nneth exact sequence 
respects the $R(G) \otimes_{\Z} R(H)$-module structure. The map $K^G_*(A) 
\otimes_{\Z} K^H_*(B) \to  K_*^{G \times H}(A \otimes  B) $ here comes from the 
natural pairing, namely if $V$ is a representation space of $G$ and $W$ is a 
representation space of $H$, and if $p \in B(V) \otimes A$ and $q \in B(W) 
\otimes B$ are invariant projections, then $[p] \otimes [q]$ gets mapped to the 
class of $p \otimes q \in B(V \otimes W) \otimes A \otimes  B$, which is a $G 
\times H$-invariant projection.
\begin{prp}\label{P_Kunneth}
	Let $G$ and $H$ be compact Lie groups. Let $A$ and $B$ be unital $C^*$-algebras endowed with actions  $\alpha \colon G \to \Aut(A)$ and $\beta \colon H \to \Aut(B)$.  Suppose that at least one of the crossed products $A \rtimes_{\alpha} G$ and $B \rtimes_{\beta} H$ is nuclear and that this pair satisfies the K{\"u}nneth formula (not taking into account the group actions). Then there is a natural exact sequence of $R(G) \otimes_{\Z} R(H)$-modules:
	\[
	0 \to K^G_*(A) \otimes_{\Z} K^H_*(B) \to  K_*^{G \times H}(A \otimes  B) \to \text{Tor}(K^G_*(A) ,K^H_{*+1}(B) ) \to 0 \, .
	\]
\end{prp}
\begin{proof}
	The fact that we have an exact sequence as claimed, thought of as an exact sequence of abelian groups, follows from the K{\"u}nneth formula using Julg's theorem \cite[Theroem 2.6.1]{phillips_book} to identify $K^G_*(A) \cong K_*(A \rtimes_{\alpha} G)$ and $K^H_*(B) \cong K_*(B \rtimes_{\beta} H)$. It remains to track the $R(G) \otimes_{\Z} R(H)$-module structure. We show that this 
	follows directly from the description of the $R(G)$-module structure on $K_*(A \rtimes_{\alpha} G)$ given in \cite[Theorem 2.7.9]{phillips_book}, which we briefly recall here. Suppose $V$ is a finite dimensional representation, and we think of $\C$ as the trivial representation. Let $p_V$ and $p_{\C}$ be the projections onto $V$ and onto $\C$ in $B(V \oplus \C)$. The maps $\varphi_V,\varphi_{\C} \colon A \to B(V \oplus \C)$, given by $\varphi_V(a) = p_V \otimes a$ and $\varphi_{\C} (a) = p_{\C} \otimes a$, induce maps $\bar{\varphi}_V,\bar{\varphi}_{\C} \colon A \rtimes_{\alpha} G \to ( B(V \oplus \C) \otimes A ) \rtimes_{\alpha} G$. For $\eta \in K_* (A \rtimes_{\alpha} G)$, we define $[V]\eta = (\bar{\varphi}_{\C})_*^{-1} (\bar{\varphi}_{V})_*$. By naturality of the K{\"u}nneth formula with respect to the maps $\varphi_V$ and $\varphi_{\C}$, as well as analogous maps for $B$, it follows that if $\eta_1 \in  K_* (A \rtimes_{\alpha} G)$ and $\eta_2 \in K_* (B \rtimes_{\beta} H )$,  and if $V$ and $W$ are finite dimensional representations of $G$ and $H$, respectively, then $[V]\eta_1 \otimes [W]\eta_2$ gets mapped to $[V \otimes W] (\eta_1 \cdot \eta_2)$, where $\eta_1 \cdot \eta_2$ here denotes the image of $\eta_1 \otimes \eta_2$ in 
	\[
	K_* ( ( A \rtimes_{\alpha} G )  \otimes ( B \rtimes_{\beta} H ) )  \cong K_*^{G \times H} ( A  \otimes  B )
	\, .
	\]
	 Likewise, the second map into the Tor group is seen by naturality to be an $R(G \times H)$-module map.
\end{proof}

\begin{prp}\label{I_2601_Prd_S1}
Let $G$ be a compact Lie group and let $d \in \Nz$.
Then there exists a simple unital AH~algebra~$A$
with no dimension growth and an action $\af \colon S^1 \times G \to \Aut (A)$
such that $\RD (\af) =  d$.
\end{prp}

\begin{proof}
	Use Proposition \ref{I_2601_S1} to construct a simple unital AH algebra $B$ with no dimension growth and an action  $\beta \colon S^1 \to \Aut(B)$ such that $\RD(\beta) = d$.
We construct a simple unital AH~algebra~$D$ with no dimension growth along with an action $\gm$ of $G$ on $D$ 
 which has the Rokhlin property, as follows. Pick a sequence $(h_1,h_2,\ldots)$ in $G$ such that each tail of this sequence is dense in $G\smallsetminus \{1\}$. We construct an inductive system as follows, with algebras as shown:
 \[
 C (G)
 \longrightarrow C (G, M_2)
 \longrightarrow C (G, M_4)
 \longrightarrow \cdots.
 \]
 The maps are all diagonal, of the form $f \mapsto \diag(f,f \circ T_n)$, with $T_n(g) = gh_n$. The maps $T_n$ are all equivariant with respect to the left translation action. As the connecting maps are all equivariant, we obtain an action $\gamma$ of $G$ on the inductive limit. Because each tail of the sequence $(h_1,h_2,\ldots)$ is dense in $G\smallsetminus \{1\}$, the inductive limit is simple. 
 
 We now set $A  = B \otimes D$, with the action $\alpha = \beta \otimes \gamma$ of $S^1 \times G$. By Corollary~\ref{C_2601_TnsR}, we have  $\RD ( \alpha ) \leq d$. For the lower bound, we claim that if $n \leq d$ then 
$I( S^1 \times G)^n K_{S^1 \times G}^*(A) \neq 0$; it then follows from Corollary~\ref{C_2601_Lower} that $\RD(\alpha) \geq d$. 
To prove the claim, we first describe the $R(S^1 \times G)$-module structure of $K_{S^1 \times G}^*(S^{2d+1} \times G)$. Notice that $K_G^0(G) = R(G)/I(G) \cong \Z$ and $K_G^1(G) = 0$. 
 Because $K_{S^1}^*(S^{2d+1})$ and $K_G^*(G)$ are both torsion free as abelian groups, by Proposition~\ref{P_Kunneth} we have 
 \[
 K_{S^1 \times G}^*(S^{2d+1} \times G) \cong K_{S^1}^*(S^{2d+1}) \otimes K_G^*(G) 
 \, ,
 \]
  so 
  \[
  K_{S^1 \times G}^0(S^{2d+1} \times G) \cong  R(S^1)/I(S^1)^d \otimes_{\Z} R(G)/I(G)
  \, .
  \]
  Note that $I(S^1 \times G) \supseteq I(S^1) \otimes_{\Z} R(G)$. Thus, if $n \leq d$ then for any positive integer $N$, we have 
  \[
  I(S^1 \times G)^n N [1]_{ K_{S^1 \times G}^*(S^{2d+1} \times G) } \neq 0
  \, .
  \]
   By the choice of the connecting maps, it then follows that $I(S^1 \times 
   G)^n K^{S^1 \times G}_0(A) \neq 0$. This proves the claim.
\end{proof}

\begin{prp}
	\label{P_2601_Prd_Z2}
	For any finite group $G$ of odd order and for any $m \in 
	\N$, there exists a simple 
	unital AF algebra $A$ with unique trace and an action $\alpha \colon \Z_2 \times G \to \Aut(A)$ such that  $m \leq 	\RD(\alpha) \leq 2m+2$. 
\end{prp}
\begin{proof}
	The proof follows the same lines as that of Proposition~\ref{I_2601_Prd_S1}. Fix $m \in \N$. 
	
	Use Proposition \ref{I_2601_Z2} to construct a simple unital AF algebra $B$ 
	with unique trace and an action  $\beta \colon \Z_2 \to \Aut(B)$ such that 
	$m \leq \RD(\beta) \leq  2m+2$.
  Construct $D$ in a manner similar to that of the proof of 
  Proposition~\ref{I_2601_Prd_S1}, as follows. Let $n$  be the cardinality of 
  $G$. We consider the inductive system 
	 \[
	C (G)
	\longrightarrow C (G) \otimes M_n
	\longrightarrow C (G) \otimes M_n \otimes M_n
	\longrightarrow \cdots
	\, ,
	\]
	with the connecting maps given as follows. Label the entries of $M_n$ by the group elements of $G$. 
	Define  $\varphi \colon C(G) \to C(G) \otimes M_n$ by 
	\[
	\varphi(f)(g) = \sum_{h \in G} f(gh) \otimes e_{hh}
	. 
	\]
	For all $k \in \Nz$, define $\varphi_{k+1,k} \colon C(G) \otimes M_n^{\otimes k} \to C(G) \otimes M_n \otimes M_n^{\otimes k}$ by 
	$\varphi_{k+1,k} = \varphi \otimes \id_{ M_n^{\otimes k} }$. 
	The inductive limit is the UHF algebra of type $n^{\infty}$. 
	The connecting maps are all equivariant with respect ot the left action of $G$ on $C(G)$, so we obtain an action of $G$ on the inductive limit with the Rokhlin property. 
	 We now set $A  = B \otimes D$, with the action $\alpha = \beta \otimes \gamma$ of $\Z_2 \times G$. By Corollary~\ref{C_2601_TnsR}, we have  $\RD ( \alpha ) \leq 2m+2$. 
	 
	 For the lower bound, we claim that 
	 $I(  \Z_2 \times G)^m  K^{\Z_2 \times G}_*(A) \neq 0$; it then  follows 
	 from Corollary~\ref{C_2601_Lower} that $\RD(\alpha) \geq m$. 
	 
	 We prove the claim. By Proposition~\ref{P_Kunneth} we have 
	 
	 \[K_{\Z_2 \times G}^*(\Z_2^{\star ( 2m+3 )} \times G) \cong K_{\Z_2}^*(\Z_2^{\star ( 2m+3 )}) \otimes K_G^*(G)
	 \, ,
	 \]
	  so 
	  \[
	  K_{\Z_2 \times G}^0(\Z_2^{\star ( 2m+3 )} \times G) \cong  R(\Z_2)/I(\Z_2)^{m+1} \otimes_{\Z} R(G)/I(G)
	  \, .
	  \]
	   We have $I(\Z_2 \times G) \supseteq I(\Z_2) \otimes_{\Z} R(G)$. 
	 Thus, using the facts about  $K$-theory of projective spaces used in the 
	 proof of Proposition \ref{I_2601_Z2}, and the fact that $R(G)$ is torsion 
	 free, we have 
	 \[
	 \left ( I(\Z_2) \otimes R(G) \right )^m 
	 K_{\Z_2 \times G}^*(\Z_2^{\star ( 2m+3 )}  \times G)  \cong \Z_2
	 \, .
	 \]
	 We conclude that 
	 $
	 I(\Z_2 \times G)^m  K_{\Z_2 \times G}^*(\Z_2^{\star ( 2m+3 )}  \times G)  
	 $
	 surjects onto $\Z_2$, and in particular is nonzero.
	  The connecting maps induce multiplication by $n$ on $K$-theory, and as 
	  $n$ is assumed to be odd, they act trivially on  $I(\Z_2 \times G)^m  
	  K_{\Z_2 \times G}^*(\Z_2^{\star ( 2m+3 )}  \times G)$. It thus follows 
	  that $I(\Z_2 \times G)^m K^{\Z_2\times G}_0(A) \neq 0$.  This completes 
	  the proof of the claim, and of the proposition.
\end{proof}

\begin{rmk}
	\label{rmk-product-action-different-dimensions}
	In the proof of Proposition \ref{P_2601_Prd_Z2} we showed that whenever $G$ is a finite group of odd order $n$, we can find actions of groups of the form $\Z_2 \times G$ on tensor products $B \otimes D$, with $B$ a simple AF algebra and $D$ the UHF algebra of type $n^{\infty}$, such that  the 
	Rokhlin dimension with commuting towers of the action of $\Z_2 \times G$  
	lies within a certain range, while the restriction to the action of the 
	second factor has the Rokhlin property. The use of $\Z_2$ allowed us to 
	explicitly give both lower and upper bounds, because of the explicit 
	knowledge of the 
	 $K$-theory of projective spaces. If we don't care about the 
	upper bounds, then an adaptation of the proof provides the following 
	statement.
	Let $p$ be a prime number, let $G$ be a finite group whose order is not 
	divisible by $p$, and let $d \in \Nz$. Denote by $n$ the order of $G$ and 
	denote by $D$ the UHF algebra of type $n^{\infty}$. Then there exists  an 
	action $\beta$ of $\Z_p$ on a simple AF algebra with unique trace and an 
	action  $\gamma$ of $G$ on $D$ such that, with $A = B \otimes D$ and 
	$\alpha = \beta \otimes \gamma$, we have $\RD(\alpha)> d$,  and the 
	restriction to the factor $G$ has the Rokhlin property. 
	
	The main change in the argument
	requires the fact that $I (\Z_p)^m / I (\Z_p)^{m + 1} \cong \Z_p$.
	As we do not know a reference, we briefly describe an argument.
	It works for any odd~$p$, not necessarily prime.
	The statement also holds when $p$ is even.
	Since we do not need that case, and the proof is more complicated,
	we omit it.
	
	We write $R (\Z_p) \cong \Z [\sigma] / (1 - \sigma^p) \Z [\sigma]$.
	Set $\lambda = 1 - \sigma$.
	Since $p$ is odd,
	\begin{equation}\label{Eq_3623_M1}
		\sum_{k = 1}^{p - 1} \binom{p}{k} (-1)^k
		= \sum_{k = 0}^{p} \binom{p}{k} (-1)^k
		= 0.
	\end{equation}
	Using the fact that $p$ is odd at the second step,
	$\sigma^p = 1$ at the third step,
	and~(\ref{Eq_3623_M1}) at the fifth step, we get
	\[
	\begin{split}
		\lambda^{p}
		& = (1 - \sigma)^p
		= 1 + \sum_{k = 1}^{p - 1} \binom{p}{k} (-1)^k \sigma^k - \sigma^p
		= \sum_{k = 1}^{p - 1} \binom{p}{k} (-1)^k (1 - \lambda)^k
		\\
		& = \sum_{k = 1}^{p - 1} \binom{p}{k} (-1)^k
		\left[ 1 + \sum_{j = 1}^{k} \binom{k}{j} (-1)^j \lambda^j \right]
		= \sum_{k = 1}^{p - 1} \binom{p}{k} (-1)^k
		\cdot \sum_{j = 1}^{k} \binom{k}{j} (-1)^j \lambda^j.
	\end{split}
	\]
	This expresses $\lambda^p$ as an integer combination
	$\lambda^p = \sum_{j = 1}^{p - 1} n_j \lambda^j$, with
	\begin{equation}\label{Eq_3623_N_1}
		n_1 = \sum_{k = 1}^{p - 1} \binom{p}{k} \cdot (-1)^{k + 1} \cdot k.
	\end{equation}

	The elements $\lambda, \lambda^2, \ldots, \lambda^{p - 1}$
	are easily seen to be linearly independent (expand $(1 - \sigma)^j$),
	and the expression $\lambda^p = \sum_{j = 1}^{p - 1} n_j \lambda^j$
	implies that their integer span is $I (\Z_p)$.
	Therefore there are direct sum decompositions, as abelian groups,
	\[
	I (\Z_p) = \bigoplus_{k = 1}^{p - 1} \Z \lambda^k
	\andeqn
	I (\Z_p)^2 = \bigoplus_{k = 2}^{p} \Z \lambda^k	\, .
	\]
	
	We now claim that 
	\[
	\sum_{k = 1}^{p - 1} \binom{p}{k} \cdot (-1)^{k + 1} \cdot k = - p
	\, .
	\]
	To prove the claim, set $h (x) = (1 - x)^p$.
	Then $\sum_{k = 1}^{p} \binom{p}{k} \cdot (-1)^{k} \cdot k = h'(1) = 0$.
	The claim follows.
	The claim and~(\ref{Eq_3623_N_1}) imply that 
	\[
	I (\Z_p)^2 = p \Z \lambda \oplus \bigoplus_{k = 2}^{p - 1} \Z \lambda^k
	\, .
	\]
	Therefore $I (\Z_p) / I (\Z_p)^2 \cong \Z_p$.
	The same argument shows that for any $m \in \N$
	we have $I (\Z_p)^m/I (\Z_p)^{m + 1} \cong \Z_p$.
\end{rmk}

\begin{exa}\label{E_2601_StrictSub}
	Strict inequality can occur in 
	Proposition~\ref{P_2531_DTens}. Our example involves actions of $\Z_6$ on a simple AF algebra with unique trace.
	Fix $d \geq 1$. 
	Using Proposition \ref{P_2601_Prd_Z2} and Remark \ref{rmk-product-action-different-dimensions}, we choose algebras and actions as follows. 
	\begin{enumerate}
		\item We choose a simple AF algebra with unique trace $B_1$ with an action $\beta_1$ of $\Z_2$, and an action $\gamma_1$ of $\Z_3$ on $D_1 = M_{3^{\infty}}$ with the Rokhlin property, such that the action $\alpha_1 = \beta_1 \otimes \gamma_1$ of $\Z_2 \times \Z_3$ on $A_1 = B_1 \otimes D_1$ satisfies $\RD(\alpha_1) > d$.
		\item We choose a simple AF algebra with unique trace $B_2$ with an action $\beta_2$ of $\Z_3$, and an action $\gamma_2$ of $\Z_2$ on $D_2 = M_{2^{\infty}}$ with the Rokhlin property, such that the action $\alpha_2 = \gamma_2 \otimes \beta_2$ of $\Z_2 \times \Z_3$ on $A_2 = D_2 \otimes B_2$ satisfies $\RD(\alpha_2) > d$.		
	\end{enumerate}
	The  action $\alpha_1 \otimes \alpha_2$ of the product $\Z_2 \times \Z_3 \cong \Z_6$ on the tensor product $A_1 \otimes A_2$ can be rewritten as a tensor product action of $\Z_6$ on $(B_1 \otimes B_2) \otimes (D_1 \otimes D_2)$. The action on the factor $D_1 \otimes D_2$ has the Rokhlin property, and therefore so does $\alpha_1 \otimes \alpha_2$. Thus $\RD ( \alpha_1 \otimes \alpha_2 ) = 0$. 
\end{exa}

\begin{qst}
	Does the phenomenon in Example \ref{E_2601_StrictSub} occur when the group is $\Z_p$ for some prime $p$?
\end{qst}

\section{The commutative case} \label{sec:commutative}
Some argument is required to show that all possible Rokhlin dimensions can be achieved for compact Lie group actions on commutative $C^*$-algebras, or more specifically, that for any $n \in \Nz$, the canonical action of $G$ on $C^*(G^{\star (n+1)}) $ has Rokhlin dimension $n$. The fact that $C(G^{\star (n+1)})$ with the canonical action is universal for Rokhlin dimension $n$ a priori only means that this action has Rokhlin dimension at most $n$. We show here that it is exactly $n$. More generally, we show that Rokhlin dimension for an action on $C(Y)$ for $Y$ compact and metrizable coincides with the so-called $G$-index of the action, shifted by $1$, a quantity already introduced in the theory of transformation groups. We recall the definition.
\begin{dfn}
	Let $G$ be a compact Lie group, and let $Y$ be a compact metrizable space equipped with a continuous action of $G$. The \emph{$G$-index} of $Y$ (as a $G$-space) is defined to be the least integer $n$ such that there exists a $G$-equivariant continuous map from $Y$ to $G^{\star n}$. We denote the $G$-index of $Y$ by $\indG (Y)$. 
\end{dfn}
We refer the reader to \cite[Section 6.2]{Matousek} for a discussion of the $G$-index, although the discussion in this reference is restricted to finite groups. 

\begin{thm}
	\label{thm:commutative-case}
	Let $G$ be a compact Lie group and let $Y$ be a compact metrizable space 
	equipped with a continuous action of $G$.  Denote the induced action of $G$ 
	on $C(Y)$ by $\alpha \colon G \to \Aut(C(Y))$. Then $\RD(\alpha) = 
	\indG(Y)-1$.
\end{thm}
\begin{proof}
	If there exists an equivariant map from $X$ to $G^{\star (n+1)}$ then it is 
	immediate from Lemma~\ref{L_2601_XRP} that $\RD(\alpha) \leq n$. We need to 
	show the reverse inequality. Set $n = \indG(Y)$. 	
	By Lemma~\ref{L_2601_XRP}, we need to show that there does not exist an 
	equivariant unital map from $C(G^{\star n})$ to the central sequence 
	algebra of $C(Y)$. 
	
	Set $X = G^{\star n}$. We denote by $\beta$ the canonical action of $G$ on 
	$C(X)$. Note that $X$ is a finite simplicial complex, endowed with a free 
	action of $G$. Use \cite[Theorem 6.1$'$]{Mostow} to pick some $N \in \N$ such 
	that there exists an action of $G$ on $\R^N$ such that $X$ embeds 
	equivariantly in $\R^N$. We fix such an embedding.  It follows from 
	\cite[Theorem 2.2]{Jaworowski} that there exists an open $G$-invariant 
	neighborhood $U$ of $X$ such that there is an equivariant retraction $r 
	\colon U \to X$. Fix such an open set $U$ and retraction $r$.
	As $X$ is an ANR, it follows from \cite[Proposition 2.11]{blackadar-shape-theory} that $C(X)$ is semiprojective in the category of commutative unital $C^*$-algebras; therefore, any unital homomorphism $C(X) \to F_{\infty}^{(\alpha)}(C(Y))$ can be lifted to a unital homomorphism $C(X) \to l^{\infty,(\alpha)}(C(Y))$. As we are starting with a  unital homomorphism $C(X) \to F_{\infty}^{(\alpha)}(C(Y))$  which is $G$-equivariant, any lifting as above will be asymptotically $G$-equivariant, meaning that for any finite set $F \subseteq C(X)$ and for any $\ep>0$ there exists a unital homomorphism $\varphi \colon C(X) \to C(Y)$ such that for all $g \in G$ and for all $f \in F$ we have $\| ( \beta_{g^{-1}} \circ \varphi \circ \alpha_g ) (f) - \varphi(f) \| < \ep$. 
	
	Let $F$ be a finite set of generators for $C(X)$. For $\varphi$ as above, 
	let us denote by $T \colon Y \to X$ the continuous function which satisfies 
	$\varphi(f) = f \circ T$. 
	We write $(g,x) \to g \cdot x$ and $(g,y) \to g \cdot y$, for $g \in G$, $x 
	\in X$, and $y \in Y$, for the actions of $G$ on $X$ and $Y$ corresponding 
	to $\alpha$ and $\beta$.
	 As the spaces are different, there is no risk of confusion. Find a 
	 sufficiently small $\ep_0>0$ such that the condition
	 $\| ( \beta_{g^{-1}} \circ \varphi \circ \alpha_g ) (f)- \varphi(f) \| < \ep_0$ for all $f \in F$ and for all $g \in G$ implies that for all $y \in Y$ there exists an open ball $D_y$ such that $T(y) \in D_y \subseteq U$ and such that    $g^{-1} \cdot T(g \cdot y) \in D_y$  for all $g \in G$. We denote the normalized Haar measure on $G$ by $dg$. Because $D_y$ is convex for any $y \in Y$, it follows that $\int_G g^{-1} \cdot T(g \cdot y) dg \in D_y \subseteq U$. Now,
	set $\tilde{T}(y) = r \left ( \int_G g^{-1} \cdot T(g \cdot y) dg \right )$. Then $\tilde{T}$ is an equivariant continuous map from $Y$ to $X$. However, such maps do not exist; see  \cite[Theorem 2.2]{Passer} and \cite[Corollary 3.1]{Volovikov}. 
\end{proof}

\begin{cor}
		Let $n \in \N$. Let $\alpha$ be the canonical action of $G$ on $C(G^{\star (n+1)})$. Then $\RD(\alpha) = n$.
\end{cor}

\begin{rmk}
	After the preprint was posted, Jianchao Wu pointed out to
	 us that Theorem~\ref{thm:commutative-case} can also be deduced from a 
	 combination of
	  Theorem 4.5 and Theorem 5.13 of \cite{GHTW}.
\end{rmk}


\begin{thebibliography}{10}
	
	\bibitem{Atyh}
	M.~F. Atiyah.
	\newblock {\em {$K$}-theory}.
	\newblock W. A. Benjamin, Inc., New York-Amsterdam, 1967.
	\newblock Lecture notes by D. W. Anderson.
	
	\bibitem{AtSgl}
	M.~F. Atiyah and G.~B. Segal.
	\newblock Equivariant {$K$}-theory and completion.
	\newblock {\em J. Differential Geometry}, 3:1--18, 1969.
	
	\bibitem{BEMSW}
	Sel\c{c}uk Barlak, Dominic Enders, Hiroki Matui, G\'{a}bor Szab\'{o}, and
	Wilhelm Winter.
	\newblock The {R}okhlin property vs. {R}okhlin dimension 1 on unital
	{K}irchberg algebras.
	\newblock {\em J. Noncommut. Geom.}, 9(4):1383--1393, 2015.
	
	\bibitem{blackadar-shape-theory}
	Bruce Blackadar.
	\newblock Shape theory for {$C^\ast$}-algebras.
	\newblock {\em Math. Scand.}, 56(2):249--275, 1985.
	
	\bibitem{Brown-continuity}
	Lawrence~G. Brown.
	\newblock Continuity of actions of groups and semigroups on {B}anach spaces.
	\newblock {\em J. London Math. Soc. (2)}, 62(1):107--116, 2000.
	
	\bibitem{CETWW}
	Jorge Castillejos, Samuel Evington, Aaron Tikuisis, Stuart White, and Wilhelm
	Winter.
	\newblock Nuclear dimension of simple {$\rm C^*$}-algebras.
	\newblock {\em Invent. Math.}, 224(1):245--290, 2021.
	
	\bibitem{EGL}
	George~A. Elliott, Guihua Gong, and Liangqing Li.
	\newblock On the classification of simple inductive limit {$C^*$}-algebras.
	{II}. {T}he isomorphism theorem.
	\newblock {\em Invent. Math.}, 168(2):249--320, 2007.
	
	\bibitem{fulton-harris}
	William Fulton and Joe Harris.
	\newblock {\em Representation theory. A first course}, volume 129 of {\em
		Graduate Texts in Mathematics}.
	\newblock Springer-Verlag, New York, 1991.
	
	\bibitem{Gardella-compact}
	Eusebio Gardella.
	\newblock Rokhlin dimension for compact group actions.
	\newblock {\em Indiana Univ. Math. J.}, 66(2):659--703, 2017.
	
	\bibitem{GHTW}
	Eusebio Gardella, Piotr M. Hajac, Mariusz Tobolski, and Jianchao Wu.
	\newblock The local-triviality dimension of actions of compact quantum 
	groups.
	\newblock Preprint, arXiv:1801.00767v2, 2018.
	
	\bibitem{GdHbSg}
	Eusebio Gardella, Ilan Hirshberg, and Luis Santiago.
	\newblock Rokhlin dimension: duality, tracial properties, and crossed products.
	\newblock {\em Ergodic Theory Dynam. Systems}, 41(2):408--460, 2021.
	
	\bibitem{GHV}
	Eusebio Gardella, Ilan Hirshberg, and Andrea Vaccaro.
	\newblock Strongly outer actions of amenable groups on {$\mathcal{Z}$}-stable
	nuclear {$C^*$}-algebras.
	\newblock {\em J. Math. Pures Appl. (9)}, 162:76--123, 2022.
	
	\bibitem{Goodearl}
	K.~R. Goodearl.
	\newblock Notes on a class of simple {$C^\ast$}-algebras with real rank zero.
	\newblock {\em Publ. Mat.}, 36(2A):637--654 (1993), 1992.
	
	\bibitem{hatcher}
	Allen Hatcher.
	\newblock {\em Algebraic topology}.
	\newblock Cambridge University Press, Cambridge, 2002.
	
	\bibitem{HP15}
	Ilan Hirshberg and N.~Christopher Phillips.
	\newblock Rokhlin dimension: obstructions and permanence properties.
	\newblock {\em Doc. Math.}, 20:199--236, 2015.
	
	\bibitem{HSWW}
	Ilan Hirshberg, G\'{a}bor Szab\'{o}, Wilhelm Winter, and Jianchao Wu.
	\newblock Rokhlin dimension for flows.
	\newblock {\em Comm. Math. Phys.}, 353(1):253--316, 2017.
	
	\bibitem{HWZ}
	Ilan Hirshberg, Wilhelm Winter, and Joachim Zacharias.
	\newblock Rokhlin dimension and {$C^*$}-dynamics.
	\newblock {\em Comm. Math. Phys.}, 335(2):637--670, 2015.
	
	\bibitem{Izm2}
	Masaki Izumi.
	\newblock Finite group actions on {$C^*$}-algebras with the {R}ohlin property.
	{II}.
	\newblock {\em Adv. Math.}, 184(1):119--160, 2004.
	
	\bibitem{Jaworowski}
	Jan~W. Jaworowski.
	\newblock Extensions of {$G$}-maps and {E}uclidean {$G$}-retracts.
	\newblock {\em Math. Z.}, 146(2):143--148, 1976.
	
	\bibitem{Kirchberg-central}
	Eberhard Kirchberg.
	\newblock Central sequences in {$C^*$}-algebras and strongly purely infinite
	algebras.
	\newblock In {\em Operator {A}lgebras: {T}he {A}bel {S}ymposium 2004}, volume~1
	of {\em Abel Symp.}, pages 175--231. Springer, Berlin, 2006.
	
	\bibitem{Matousek}
	Ji\v{r}\'{\i} Matou\v{s}ek.
	\newblock {\em Using the {B}orsuk-{U}lam theorem}.
	\newblock Universitext. Springer-Verlag, Berlin, 2003.
	\newblock Lectures on topological methods in combinatorics and geometry,
	Written in cooperation with Anders Bj\"{o}rner and G\"{u}nter M. Ziegler.
	
	\bibitem{Mostow}
	G.~D. Mostow.
	\newblock Equivariant embeddings in {E}uclidean space.
	\newblock {\em Ann. of Math. (2)}, 65:432--446, 1957.
	
	\bibitem{Passer}
	Benjamin Passer.
	\newblock Free actions on {$C^*$}-algebra suspensions and joins by finite
	cyclic groups.
	\newblock {\em Indiana Univ. Math. J.}, 67(1):187--203, 2018.
	
	\bibitem{phillips_book}
	N.~Christopher Phillips.
	\newblock {\em Equivariant {$K$}-theory and freeness of group actions on
		{$C^*$}-algebras}, volume 1274 of {\em Lecture Notes in Mathematics}.
	\newblock Springer-Verlag, Berlin, 1987.
	
	\bibitem{szabo-model-actions}
	G\'{a}bor Szab\'{o}.
	\newblock Rokhlin dimension: absorption of model actions.
	\newblock {\em Anal. PDE}, 12(5):1357--1396, 2019.
	
	\bibitem{SWZ}
	G\'{a}bor Szab\'{o}, Jianchao Wu, and Joachim Zacharias.
	\newblock Rokhlin dimension for actions of residually finite groups.
	\newblock {\em Ergodic Theory Dynam. Systems}, 39(8):2248--2304, 2019.
	
	\bibitem{Volovikov}
	A.~Yu. Volovikov.
	\newblock Coincidence points of mappings of {$\mathbb{Z}^n_p$}-spaces.
	\newblock {\em Izv. Ross. Akad. Nauk Ser. Mat.}, 69(5):53--106, 2005.
	
	\bibitem{wouters2023equivariant}
	Lise Wouters.
	\newblock Equivariant $\mathcal{Z}$-stability for single automorphisms on
	simple {$C^*$}-algebras with tractable trace simplices.
	\newblock Preprint, arXiv:2105.04469v3, 2021.
	
\end{thebibliography}
\end{document}